\documentclass[a4paper,11pt,pdflatex]{article}

\usepackage{mathtools} 
\usepackage{amssymb,amsfonts,amsthm} 

\usepackage{graphicx}
\usepackage{tikz}
\usepackage{url}

\usepackage{ascmac}    
\usepackage{fancybox}  
\usepackage{float}     

\usepackage[shortlabels]{enumitem} 

\usepackage[setpagesize=false,bookmarks=true,bookmarksdepth=tocdepth,bookmarksnumbered=true,colorlinks=true,allcolors=blue,pdftitle={},pdfsubject={},pdfauthor={},pdfkeywords={}]{hyperref}
\usepackage[nameinlink,capitalize]{cleveref}
\usepackage{aliascnt}
\usepackage{autonum} 

\DeclareMathOperator{\Aut}{Aut}
\DeclareMathOperator{\id}{id}
\DeclareMathOperator{\col}{col}

\newcommand{\R}{\mathbb{R}}
\newcommand{\N}{\mathbb{N}}
\newcommand{\Z}{\mathbb{Z}}

\newcommand{\isomto}{\xrightarrow{\sim}}
\newcommand{\co}{\colon}

\numberwithin{equation}{section}

\theoremstyle{plain}
\newtheorem{thm}{Theorem}[section]

\newaliascnt{lem}{thm}
\newtheorem{lem}[lem]{Lemma}
\aliascntresetthe{lem}

\newaliascnt{cor}{thm}

\aliascntresetthe{cor}

\newaliascnt{prop}{thm}
\newtheorem{prop}[prop]{Proposition}
\aliascntresetthe{prop}

\newaliascnt{conj}{thm}
\newtheorem{conj}[conj]{Conjecture}
\aliascntresetthe{conj}

\newaliascnt{ques}{thm}
\newtheorem{ques}[ques]{Question}
\aliascntresetthe{ques}

\newaliascnt{fact}{thm}

\aliascntresetthe{fact}

\theoremstyle{definition}
\newaliascnt{dfn}{thm}
\newtheorem{dfn}[dfn]{Definition}
\aliascntresetthe{dfn}

\newaliascnt{prob}{thm}

\aliascntresetthe{prob}

\newaliascnt{ex}{thm}
\newtheorem{ex}[ex]{Example}
\aliascntresetthe{ex}

\newaliascnt{obs}{thm}

\aliascntresetthe{obs}

\theoremstyle{remark}
\newaliascnt{rem}{thm}
\newtheorem{rem}[rem]{Remark}
\aliascntresetthe{rem}

\crefname{equation}{Equation}{Equations}
\crefname{figure}{Figure}{Figures}
\crefname{table}{Table}{Tables}
\crefname{section}{Section}{Sections}
\crefname{subsection}{Subsection}{Subsections}
\crefname{appendix}{Appendix}{Appendices}

\crefname{thm}{Theorem}{Theorems}
\crefname{lem}{Lemma}{Lemmas}
\crefname{cor}{Corollary}{Corollaries}
\crefname{prop}{Proposition}{Propositions}
\crefname{conj}{Conjecture}{Conjectures}
\crefname{fact}{Fact}{Facts}
\crefname{dfn}{Definition}{Definitions}
\crefname{prob}{Problem}{Problems}
\crefname{ex}{Example}{Examples}
\crefname{obs}{Observation}{Observations}
\crefname{rem}{Remark}{Remarks}

\crefformat{equation}{Equation~(#2#1#3)}
\crefrangeformat{equation}{Equations~(#3#1#4)--(#5#2#6)}
\crefmultiformat{equation}{Equations~(#2#1#3)}{, (#2#1#3)}{, (#2#1#3)}{, (#2#1#3)}

\allowdisplaybreaks

\setlist[enumerate,1]{label=(\arabic*)}
\setlist[enumerate,2]{label=(\alph*)}
\setlist[enumerate,3]{label=(\roman*)}
\setlist[enumerate,4]{label=(\Alph*)}

\begin{document}

\title{A Non-Multiplicable Upho Poset Constructed from the Petersen Graph}
\author{Ryunosuke Matsuoka}
\date{\today}
\maketitle

\begin{abstract}
An upper homogeneous (upho) poset is a poset whose every principal filter is
isomorphic to the whole poset. Fu--Peng--Zhang conjectured that every finitary
upho poset admits a compatible left-cancellative, invertible-free monoid
structure whose left-divisibility order coincides with the given order.

We disprove this conjecture. For every vertex-transitive graph $G$, we construct
a finitary upho poset $P(G,v_0)$ from walks starting at a
fixed vertex $v_0$. Applying this construction to the Petersen graph, we show
that multiplicability of $P(G,v_0)$ would force the automorphism group of $G$ to
contain a regular subgroup. This would imply that $G$ is a Cayley graph, contradicting the fact that the Petersen graph is
not Cayley. Hence $P(G,v_0)$ is a non-multiplicable finitary upho poset.

We also show that the analogous poset associated with the line graph of the
Petersen graph is multiplicable, demonstrating that non-Cayleyness of the
underlying graph alone does not determine multiplicability.
\end{abstract}

\tableofcontents

\newpage

\section{Introduction}

A poset $P$ is called \textbf{upho (upper homogeneous)} if for any element $a \in P$, the principal filter $U(a) := \{ x \in P \mid x \geq a \}$ is order-isomorphic to $P$. 
This class of infinite posets was first introduced by Richard Stanley in his talks (see, e.g., \cite{Sta20}) related to his research on Stern's diatomic array, an array of integers analogous to Pascal's triangle.
The first detailed study of upho posets was conducted by Gao, Guo, Seetharaman, and Seidel \cite{GGSS22}, who primarily investigated their rank-generating functions. Notably, their paper also introduced the idea of constructing upho posets from left-cancellative, invertible-free (LCIF) monoids. 
This monoid-theoretic perspective was subsequently expanded upon by Fu, Peng, and Zhang \cite{FPZ24}. An interesting conjecture regarding the structure of upho posets was proposed in their work. This conjecture states that when an upho poset $P$ is finitary, $P$ can be endowed with a monoid structure such that its left divisibility relation $a \leq b \iff \exists c, a * c = b$ coincides with the partial order of $P$. 
This conjecture was originally proposed by Fu--Peng--Zhang. Independently, Hidetaka Inagaki arrived at the same conjecture, which motivated the author to initiate this project.

The purpose of this paper is to refute this conjecture by providing a concrete counterexample based on reachability in a vertex-transitive graph. Let $G$ be the Petersen graph, and define the set $P$ as
\begin{align}
    P := \{ (v,n) \in V(G) \times \N \mid \text{there is a walk of length $n$ from $v_0$ to $v$} \}.
\end{align}
Furthermore, we define a partial order $\leq$ on $P$ by
\begin{align}
    (v,n) \leq (w,m) \Leftrightarrow \text{$n \leq m$ and there is a walk of length $m-n$ from $v$ to $w$}.
\end{align}
Then, $(P, \leq)$ is a finitary upho poset. If $(P, \leq)$ were multiplicable (i.e., admitted a compatible monoid structure), then for any $a = (u,n) \in P$, the multiplication by $a$ can be expressed as follows:
\begin{align}
    a * (v,m) = (\phi_{a,m}(v), m + n)
\end{align}
Using a rigidity property of the Petersen graph (namely, that two vertex bijections preserving adjacency in the sense $v \sim w \iff \phi(v) \sim \psi(w)$ must coincide), we show that
$\phi_{a,m}$ is independent of $m$ and is the restriction of an automorphism $\phi_a$ of $G$.
Using this $\phi_a$, we construct a subgroup $H$ of the automorphism group $\Aut(G)$ that acts regularly on the vertex set $V(G)$. 
This would imply that $G$ is a Cayley graph. However, since the Petersen graph is not a Cayley graph, $(P, \leq)$ is not multiplicable.

For any vertex-transitive graph $G$, an upho poset can be constructed in a similar manner. This method provides a rich family of concrete examples of finitary upho posets and is expected to reveal many connections among posets, monoids, and graphs.
Furthermore, we discover that the upho poset constructed from the line graph of the Petersen graph is indeed multiplicable. This shows that the non-Cayleyness of the underlying graph alone does not determine the multiplicability of the associated upho poset.

Finally, we note that our counterexample constructed from the Petersen graph is not a lattice. 
While the Fu--Peng--Zhang conjecture fails for general finitary upho posets, the analogue for upho lattices remains open, which we formally state as \cref{conj:upho-lattice-multiplicable} along with its connection to the classification theory by Hopkins and Lewis.

\section{Preliminaries}

\subsection{Upho Posets and LCIF Monoids}

A monoid is an algebraic structure $(M, *, e)$ consisting of a set, an associative binary operation $*$, and an identity element $e$ (satisfying $e*a = a*e = a$).

\begin{dfn}
    A monoid $(M, *, e)$ is called \textbf{LCIF (left-cancellative invertible-free)} if it satisfies the following conditions:
    \begin{enumerate}
        \item For any $a, b, c \in M$, $a*b = a*c$ implies $b = c$ (left-cancellative).
        \item For any $a, b \in M$, $a*b = e$ implies $a = b = e$ (invertible-free).
    \end{enumerate}
\end{dfn}

\begin{ex}
    \begin{enumerate}
        \item The set of non-negative integers $\N$ and its direct product $\N^n$ are LCIF monoids under usual addition.
        \item $\Z_{>0}$ is an LCIF monoid under multiplication.
        \item The free monoids generated by any set are LCIF monoids.
        \item The non-commutative monoid defined by generators $a,b,c$ and the relation $ac=bc$ is an LCIF monoid, but it is not right-cancellative.
    \end{enumerate}
\end{ex}

A characteristic of LCIF monoids is that the \textbf{left divisibility relation} on $M$,
\begin{align}
    a \leq b \iff \exists c \in M, a * c = b
\end{align}
defines a poset (partially ordered set). That is, for any $a, b, c \in M$, the following conditions hold:
\begin{enumerate}
    \item $a \leq a$.
    \item If $a \leq b$ and $b \leq c$, then $a \leq c$.
    \item If $a \leq b$ and $b \leq a$, then $a = b$.
\end{enumerate}
Furthermore, this order has the following property.

\begin{dfn}
    A poset $(P, \leq)$ is \textbf{upho (upper homogeneous)} if for any $a \in P$, there exists an order isomorphism:
    \begin{align}
        \psi_a \co P \isomto U(a) := \{ x \in P \mid x \geq a \}.
    \end{align}
\end{dfn}

\begin{ex}
    \begin{enumerate}
        \item $\N$ and $\N^n$ are upho posets under the usual order and the product order, respectively.
        \item $\Z_{>0}$ is an upho poset under the divisibility relation.
        \item The infinite binary tree is an upho poset; it is the left-divisibility poset of the free monoid generated by two elements.
        \item $\Z$ is not an upho poset under the usual order.
    \end{enumerate}
\end{ex}

Note that an upho poset automatically has a minimum element. 

\begin{dfn}
    An upho poset $P$ is \textbf{multiplicable} if there exists an LCIF monoid structure on $P$ such that the left divisibility relation coincides with the partial order of $P$.
    In this case, the identity element of the monoid is automatically the minimum element of $P$.
\end{dfn}

Let $P$ be a poset. For $x, y \in P$, we write $x \lessdot y$ if $x < y$ and there is no element $z$ such that $x < z < y$ (i.e., $y$ covers $x$).
For $a \in P$, its height is defined as
\begin{align}
    \mathrm{ht}(x) := \sup \{ n \in \N \mid \text{there exists a chain } x_0 < x_1 < \cdots < x_n = a \}.
\end{align}
$P$ is called \textbf{finitary} if for any $a \in P$, $\mathrm{ht}(a) < \infty$. Additionally, a finitary poset $P$ with a minimum element $e$
is \textbf{$\N$-graded} if there exists a map $\rho \co P \to \N$ such that 
\begin{align}
    x \lessdot y \Rightarrow \rho(x) + 1 = \rho(y), \quad \rho(e) = 0.
\end{align}
$\rho(x)$ is called the \textbf{rank} of $x$.
An $\N$-graded poset $P$ is \textbf{finite-type} if for any $n \in \N$, the set $P_n := \{ x \in P \mid \rho(x) = n \}$ is finite.
We say that an LCIF monoid $M$ is an \textbf{LCH (left-cancellative homogeneous)} if its left divisibility relation gives $\N$-graded upho poset.

\subsection{Colorings of Upho Posets}

One of the key concepts in \cite{FPZ24} is the notion of a \textbf{coloring} of an upho poset, which serves as a combinatorial tool to characterize multiplicability. 
We don't need this notion to construct a counterexample, but the original conjecture was formulated in terms of colorings, so we introduce it here.

For a finitary upho poset $P$,
\begin{align}
    E_P := \{ (x,y) \in P \times P \mid x \lessdot y \}
\end{align}
is the set of covering relations in $P$, and
\begin{align}
    A_P := \{ a \in P \mid e \lessdot a \}
\end{align}
is the set of \textbf{atoms} in $P$. 

\begin{dfn}
    Let $P$ be a finitary upho poset. A map
    \begin{align}
        \col\co E_P \to A_P
    \end{align}
    is called a \textbf{coloring} of $P$ if it satisfies the following conditions:
    \begin{enumerate}
        \item For any atom $a\in A_P$,
        \begin{align}
            \col(e,a)=a.
        \end{align}
        \item For any $s\in P$, there exists an order isomorphism
        \begin{align}
            \psi_s \co P \isomto U(s)
        \end{align}
        such that, for every covering $(x,y)\in E_P$,
        \begin{align}
            \col(x,y) = \col(\psi_s(x), \psi_s(y)).
        \end{align}
    \end{enumerate}
    It can be shown that such a $\psi_s$ is unique for each $s$.
\end{dfn}

\begin{prop}\label{prop:regular_multiplicable}
    Let $P$ be a finitary upho poset. Then $P$ admits a coloring if and only if
    $P$ is multiplicable.
\end{prop}

\begin{proof}
    This is essentially a reformulation of \cite[Theorem 5]{FPZ24}.

    Suppose first that $P$ admits a coloring. For each $s \in P$, let
    $\psi_s \co P \isomto U(s)$ be the color-preserving isomorphism. Define a binary operation on $P$ by
    \begin{align}
        s * t := \psi_s(t).
    \end{align}
    Since $\psi_e=\id_P$ and $\psi_s(e)=s$, the element $e$ is the identity for
    $*$. Moreover, both $\psi_{r*s}$ and $\psi_r\circ \psi_s$ are
    color-preserving isomorphisms $P \isomto U(r*s)$. By uniqueness, they are
    equal. Hence
    \begin{align}
        (r*s)*t = \psi_{r*s}(t) = \psi_r(\psi_s(t)) = r*(s*t),
    \end{align}
    so $(P,*,e)$ is a monoid.

    This monoid is left-cancellative, since
    \begin{align}
        s*x=s*y \implies \psi_s(x)=\psi_s(y) \implies x=y.
    \end{align}
    It is also invertible-free: if $x*y=e$, then $e\in U(x)$, hence $x\leq e$,
    so $x=e$, and then $y=e$. Finally, if $s \leq u$, then $u \in U(s)$, so $u=\psi_s(\psi_s^{-1}(u)) = s*\psi_s^{-1}(u)$. 
    Conversely, if $u=s*t$ for some $t$, then $u \in U(s)$, so $s \leq u$. Thus the original order coincides with the left divisibility order, and
    $P$ is multiplicable.

    Conversely, suppose that $P$ is multiplicable. Thus $P$ has an LCIF monoid
    structure $(P,*,e)$ whose left divisibility order is the given order. For
    each $s\in P$, define
    \begin{align}
        \psi_s \co P \to U(s),\qquad \psi_s(t):=s*t.
    \end{align}
    This map is bijective because every element of $U(s)$ is uniquely of the
    form $s*t$ by left divisibility and left-cancellativity. It is also an order
    isomorphism. For a covering edge $(x,y)\in E_P$, there is a unique atom $a \in A_P$ such that 
    \begin{align}
        x * a = y.
    \end{align}
    Define
    \begin{align}
        \col(x,y):=a.
    \end{align}
    Then $\col(e,a)=a$ for every atom $a$.

    It remains to check compatibility. Let $(x,y)\in E_P$ and write
    \begin{align}
        x*a=y,\qquad a=\col(x,y).
    \end{align}
    Then
    \begin{align}
        \psi_s(x)*a = (s*x)*a = s*(x*a) = s*y = \psi_s(y).
    \end{align}
    Hence
    \begin{align}
        \col(\psi_s(x),\psi_s(y)) = a = \col(x,y).
    \end{align}
    Therefore the coloring is compatible, and $P$ admits a coloring.
\end{proof}

The aim of this paper is to construct a counterexample to the conjecture proposed in \cite{FPZ24}.

\begin{conj}[{\cite{FPZ24}}] \label{conj:main-finite-type}
    Every finitary upho poset, in particular every finite-type $\N$-graded upho poset, has a coloring.
\end{conj}

This statement proposed as Conjecture 3.12 and Conjecture 7.1 in the preprint version \cite{FPZ24a},
 which corresponds to the surjectivity of the forgetful mapping $\mathfrak{F}$ in Section 3.3 of the FPSAC version \cite{FPZ24}.
By \cref{prop:regular_multiplicable}, this conjecture is equivalent to the following statement.

\begin{conj}
    Every finitary $\N$-graded upho poset, in particular every finite-type $\N$-graded upho poset, is multiplicable.
\end{conj}

\subsection{Graph Theory}

A \textbf{graph} $G = (V, E)$ is a pair consisting of a finite set $V$ and a set $E \subset \binom{V}{2}$. Elements of $V$ are called vertices, and elements of $E$ are called edges. When $\{u,v\} \in E$, we write $u \sim v$.

An \textbf{isomorphism} between graphs $G$ and $H$ is a bijection $\sigma \co V(G) \isomto V(H)$ such that for any $u, v \in V(G)$, $u \sim v \iff \sigma(u) \sim \sigma(v)$. In particular, an isomorphism from $G$ to itself is called an \textbf{automorphism}, and the set of all automorphisms is denoted by $\Aut(G)$. $G$ is \textbf{vertex-transitive} if $\Aut(G)$ acts transitively on $V$.

A \textbf{walk} of length $n$ from vertex $v$ to $w$ is a sequence of vertices
\begin{align}
    v = v_0 \sim v_1 \sim \cdots \sim v_n = w.
\end{align}
$G$ is \textbf{connected} if for any vertices $v, w$, there exists a walk from $v$ to $w$.

For a vertex $v$, 
\begin{align}
    N_G(v) := \{ w \in V \mid w \sim v \}
\end{align}
is called the \textbf{neighborhood} of $v$. When $N_G(v) = N_G(w) \Rightarrow v = w$ holds, $G$ is called \textbf{twin-free}.
For a graph $G = (V,E)$, the matrix $A = (a_{uv})_{u,v \in V}$ indexed by $V$ is defined as
\begin{align}
    a_{uv} = \begin{cases}
        1 & u \sim v \\ 0 & \text{otherwise.}
    \end{cases}
\end{align}
This matrix $A$ is called the \textbf{adjacency matrix} of $G$.

A graph $G$ is \textbf{bipartite} if there exists a partition $V = V_1 \sqcup V_2$ of the vertex set such that no two vertices within $V_1$ or within $V_2$ are adjacent.

\begin{ex}\label{ex:petersen}
    Let $X := \{1,2,3,4,5\}$, and define a graph $G = (V,E)$ by
    \begin{align}
        V = \binom{X}{2}, \quad u \sim v \iff u \cap v = \emptyset.
    \end{align}
    Then $G$ is a connected, vertex-transitive graph that is twin-free and not bipartite. This graph $G$ is called the \textbf{Petersen graph}.
\end{ex}

\begin{figure}[H]
    \centering
    \begin{tikzpicture}[
        vertex/.style={
            circle, 
            draw=blue!70!black, 
            fill=blue!5, 
            thick,
            minimum size=26pt, 
            inner sep=0pt, 
            font=\small
        },
        edge/.style={
            thick, 
            draw=gray!70!black
        },
        scale=0.8
    ]

        \node[vertex] (u0) at (90:3.8)  {$\{1,2\}$};
        \node[vertex] (u1) at (162:3.8) {$\{3,4\}$};
        \node[vertex] (u2) at (234:3.8) {$\{1,5\}$};
        \node[vertex] (u3) at (306:3.8) {$\{2,3\}$};
        \node[vertex] (u4) at (18:3.8)  {$\{4,5\}$};

        \node[vertex] (v0) at (90:1.9)  {$\{3,5\}$};
        \node[vertex] (v1) at (162:1.9) {$\{2,5\}$};
        \node[vertex] (v2) at (234:1.9) {$\{2,4\}$};
        \node[vertex] (v3) at (306:1.9) {$\{1,4\}$};
        \node[vertex] (v4) at (18:1.9)  {$\{1,3\}$};

        \draw[edge] (u0) -- (u1);
        \draw[edge] (u1) -- (u2);
        \draw[edge] (u2) -- (u3);
        \draw[edge] (u3) -- (u4);
        \draw[edge] (u4) -- (u0);

        \draw[edge] (v0) -- (v2);
        \draw[edge] (v2) -- (v4);
        \draw[edge] (v4) -- (v1);
        \draw[edge] (v1) -- (v3);
        \draw[edge] (v3) -- (v0);

        \draw[edge] (u0) -- (v0);
        \draw[edge] (u1) -- (v1);
        \draw[edge] (u2) -- (v2);
        \draw[edge] (u3) -- (v3);
        \draw[edge] (u4) -- (v4);

    \end{tikzpicture}
    \caption{The Petersen graph} \label{fig:petersen}
\end{figure}
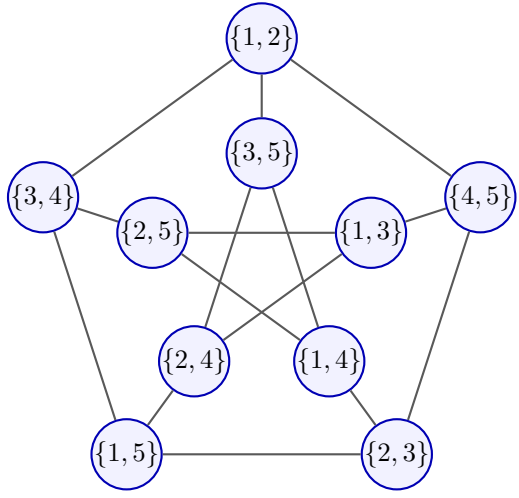

\subsection{Cayley Graphs}

\begin{dfn}
    Let $H$ be a finite group and $S \subset H$ be a generating set of $H$ such that
    \begin{align}
        e \notin S, \quad S = S^{-1}.
    \end{align}
    The graph $\mathrm{Cay}(H,S) = (V,E)$ is defined as
    \begin{align}
        V := H, \quad E := \{\{h, hs\} \mid h \in H, s \in S\}.
    \end{align}
    Then $\mathrm{Cay}(H,S)$ (or graphs isomorphic to it) is called the \textbf{Cayley graph} of the group $H$ with respect to the generating set $S$. Since $S$ is a generating set, $\mathrm{Cay}(H,S)$ is connected. It is also vertex-transitive.
\end{dfn}

The following characterization of Cayley graphs has long been known.

\begin{prop}[Sabidussi's Theorem, \cite{Sa58}] \label{prop:sabidussi}
    A graph $G$ is isomorphic to a Cayley graph of a group $H$ if and only if there exists a regular action of $H$ on $G$ (i.e., a regular subgroup $H \le \Aut(G)$).
\end{prop}

\begin{proof}
    $\mathrm{Cay}(H,S)$ has a natural regular action by $H$. Conversely, when a graph $G$ admits a regular action of a group $H$, if we define $S := \{ s \in H \mid s \cdot v_0 \sim v_0 \}$, then
    \begin{align}
        \mathrm{Cay}(H,S) \isomto G, \quad h \mapsto h \cdot v_0
    \end{align}
    is a graph isomorphism.
\end{proof}

\begin{prop}[{\cite{God80}, \cite[Lemma 3.1.3]{GR01}}]\label{prop:petersen_is_not_cayley}
    The Petersen graph is not a Cayley graph.
\end{prop}

\begin{proof}
    The natural action of the symmetric group $\mathfrak{S}_5$ on $V$ given by
    \begin{align}
        \mathfrak{S}_5 \curvearrowright V; \quad \sigma(\{i,j\}) = \{\sigma(i), \sigma(j)\}
    \end{align}
    induces a group homomorphism $\mathfrak{S}_5 \to \Aut(G)$. It is known that this is an isomorphism. That is,
    \begin{align}
        \Aut(G) \cong \mathfrak{S}_5.
    \end{align}
    Assume there exists a regular subgroup $H$ of $\Aut(G) = \mathfrak{S}_5$; then $|H| = |V| = 10$. 
    There are exactly six subgroups of order 10 in $\mathfrak{S}_5$, and it is known that they are all conjugate. Thus, without loss of generality, we may assume
    \begin{align}
        H = \langle \sigma, \tau \rangle, \quad \sigma = (1\ 2\ 3\ 4\ 5), \quad \tau = (1\ 4)(2\ 3).
    \end{align}
    Considering the orbits of $\sigma$, they split into the outer and inner vertices of the graph:
    \begin{align}
        O_1 &= \{\{1,2\}, \{2,3\}, \{3,4\}, \{4,5\}, \{1,5\}\} \\
        O_2 &= \{\{1,3\}, \{2,4\}, \{3,5\}, \{1,4\}, \{2,5\}\}.
    \end{align}
    Therefore, for the action of $H$ on $V$ to be transitive, $\tau$ must map elements of $O_1$ to elements of $O_2$. However, since the action of $\tau$ is
    \begin{align}
        \{1,2\} \leftrightarrow \{3,4\}, \quad \{4,5\} \leftrightarrow \{1,5\}, \quad \{2,3\} \leftrightarrow \{2,3\},
    \end{align}
    $\tau$ maps elements of $O_1$ to $O_1$. This is a contradiction.
\end{proof}

\section{Construction of the Counterexample}

\subsection{Upho Posets Constructed from Graphs}

\begin{dfn}
    Let $G = (V,E)$ be a graph. For $v_0 \in V$, the set $P = P(G,v_0)$ is defined as
    \begin{align}
        P := \{ (v,n) \in V \times \N \mid  \text{there is a walk of length $n$ from $v_0$ to $v$} \}.
    \end{align}
    Furthermore, define a partial order $\leq$ on $P$ by
    \begin{align}
        (v,n) \leq (w,m) \Leftrightarrow n \leq m \text{ and there is a walk of length $m-n$ from $v$ to $w$}.
    \end{align}
    $P$ is a poset with $(v_0,0)$ as its minimum element.
\end{dfn}

\begin{ex}
    Let the graph $G = (V,E)$ be the complete graph $K_3$. That is,
    \begin{align}
        V = \{v_0, v_1, v_2\}, \quad E = \{\{v_0,v_1\}, \{v_0,v_2\}, \{v_1,v_2\}\}.
    \end{align}
    $P = P(G,v_0)$ becomes a poset with the following Hasse diagram.

    \begin{figure}[H]
        \centering
        \begin{tikzpicture}[
            vertex/.style={
                circle, 
                draw=blue!70!black, 
                fill=blue!5, 
                thick,
                minimum size=32pt, 
                inner sep=0pt, 
                font=\small
            },
            edge/.style={
                thick, 
                draw=gray!70!black
            }
        ]

            \node[vertex] (v00) at (0, 0) {$(v_0,0)$};

            \node[vertex] (v11) at (-2, 1.5) {$(v_1,1)$};
            \node[vertex] (v21) at (2, 1.5) {$(v_2,1)$};

            \node[vertex] (v12) at (-2, 3.0) {$(v_1,2)$};
            \node[vertex] (v02) at (0, 3.0) {$(v_0,2)$};
            \node[vertex] (v22) at (2, 3.0) {$(v_2,2)$};

            \node[vertex] (v13) at (-2, 4.5) {$(v_1,3)$};
            \node[vertex] (v03) at (0, 4.5) {$(v_0,3)$};
            \node[vertex] (v23) at (2, 4.5) {$(v_2,3)$};

            \node[vertex] (v14) at (-2, 6.0) {$(v_1,4)$};
            \node[vertex] (v04) at (0, 6.0) {$(v_0,4)$};
            \node[vertex] (v24) at (2, 6.0) {$(v_2,4)$};
            
            \node (dots1) at (-2, 7.5) {$\vdots$};
            \node (dots0) at (0, 7.5) {$\vdots$};
            \node (dots2) at (2, 7.5) {$\vdots$};

            \draw[edge] (v00) -- (v11);
            \draw[edge] (v00) -- (v21);

            \draw[edge] (v11) -- (v02);
            \draw[edge] (v11) -- (v22);
            \draw[edge] (v21) -- (v02);
            \draw[edge] (v21) -- (v12);

            \draw[edge] (v02) -- (v13);
            \draw[edge] (v02) -- (v23);
            \draw[edge] (v12) -- (v03);
            \draw[edge] (v12) -- (v23);
            \draw[edge] (v22) -- (v03);
            \draw[edge] (v22) -- (v13);

            \draw[edge] (v03) -- (v14);
            \draw[edge] (v03) -- (v24);
            \draw[edge] (v13) -- (v04);
            \draw[edge] (v13) -- (v24);
            \draw[edge] (v23) -- (v04);
            \draw[edge] (v23) -- (v14);
            
            \draw[edge, dashed] (v04) -- (dots1);
            \draw[edge, dashed] (v04) -- (dots2);
            \draw[edge, dashed] (v14) -- (dots0);
            \draw[edge, dashed] (v14) -- (dots2);
            \draw[edge, dashed] (v24) -- (dots0);
            \draw[edge, dashed] (v24) -- (dots1);

        \end{tikzpicture}
        \caption{The Hasse diagram of $P(K_3,v_0)$} \label{fig:K3_poset}
    \end{figure}
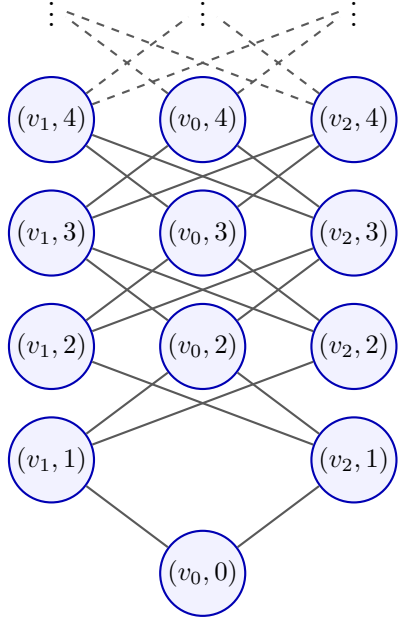
\end{ex}

\begin{rem}
    We can easily verify that for any graph $G$ and $(v,n), (w,m) \in P(G,v_0)$, the covering relation $\lessdot$ is characterized by
    \begin{align}
        (v,n) \lessdot (w,m) \Leftrightarrow m = n + 1 \text{ and } v \sim w.
    \end{align}
\end{rem}

\begin{prop}
    If $G$ is a vertex-transitive graph, then $P$ is a finite-type $\N$-graded upho poset.
\end{prop}

\begin{proof}
    Take any $a = (u,n) \in P$. Since $G$ is vertex-transitive, there exists $\sigma \in \Aut(G)$ satisfying $\sigma(v_0) = u$. Then, the map
    \begin{align}
        \psi_a \co P \to U(a); \quad (v,m) \mapsto (\sigma(v),n+ m)
    \end{align}
    is an order isomorphism. $P$ is a finitary poset since any chain from $(v_0,0)$ to $a$ has length at most $n$.
    If we define $\rho \co P \to \N$ by $\rho(v,n) = n$, then $x \lessdot y$ implies $\rho(x) + 1 = \rho(y)$ and $\rho(v_0,0) = 0$, so $P$ is $\N$-graded. Finally,
    the set $P_n = \{ (v,n) \mid (v,n) \in P \}$ is finite since $V(G)$ is finite. Therefore, $P$ is a finite-type $\N$-graded upho poset.
\end{proof}

\begin{rem}
    If $G$ is vertex-transitive, $P(G,v_0)$ does not depend on the choice of $v_0$ up to order-isomorphism. 
\end{rem}

\begin{ex}
    Let $G = \mathrm{Cay}(H,S)$ be a Cayley graph. Then $P = P(G,e)$ is the following set:
    \begin{align}
        P = \{ (v,n) \in H \times \N \mid v \in S^n \},
    \end{align}
    and the partial order $\leq$ on $P$ is given by
    \begin{align}
        (v,n) \leq (w,m) \Leftrightarrow n \leq m \text{ and } v^{-1} w \in S^{m-n}.
    \end{align}
    Using the operation of $H$, we can simply define a product $*$ on $P(G,e)$ as follows:
    \begin{align}
        (v,n) * (w,m) := (vw, n + m).
    \end{align}
    This shows that $P$ is multiplicable when $G$ is a Cayley graph.
\end{ex}

\subsection{Construction of a Non-multiplicable Upho Poset}

In this subsection, we will prove the following theorem, which provides a counterexample to \cref{conj:main-finite-type}.
This is the main result of this paper.
\begin{thm}\label{thm:petersen-nonmultiplicable}
    Let $G = (V,E)$ be the Petersen graph. Then the upho poset $P = P(G,v_0)$ is not multiplicable.
\end{thm}

\begin{proof}
First we will show the following rigidity property of the Petersen graph.
\begin{lem}\label{lem:petersen-rigidity}
    If bijections between vertex sets
    \begin{align}
        \phi, \psi \co V \to V
    \end{align}
    satisfy $v \sim w \iff \phi(v) \sim \psi(w)$ for any $v,w \in V$, then $\phi = \psi$.
\end{lem}

\begin{proof}
    Let $A$ be the adjacency matrix of $G$. Since
    \begin{align}
        \forall v,w, \quad A_{vw} = A_{\phi(v)\psi(w)},
    \end{align}
    we define permutation matrices $P, Q$ as
    \begin{align}
        P_{vw} = \delta_{\phi(v)w}, \quad Q_{vw} = \delta_{\psi(v)w}.
    \end{align}
    Thus, $A = PAQ^\top$, giving $PA = AQ$. By symmetry, $QA = AP$ also holds, so
    \begin{align}
        (P-Q)A = A(Q-P).
    \end{align}
    Setting $B := P-Q$, we have $AB = -BA$. Therefore, for an eigenvector $x$ of $A$ with eigenvalue $\lambda$,
    \begin{align}
        ABx = -BAx = -\lambda Bx,
    \end{align}
    which implies that $Bx$ belongs to the eigenspace of $-\lambda$. On the other hand, the eigenvalues of the adjacency matrix $A$ of the Petersen graph are $3, 1^{(5)}, (-2)^{(4)}$, so $\lambda$ and $-\lambda$ are never both eigenvalues. Thus, $Bx=0$. Since $A$ is a real symmetric matrix and is diagonalizable, applying this reasoning to all eigenspaces yields $B = 0$. Therefore, $\phi = \psi$.
\end{proof}

We now return to the main proof. For each $n \in \N$, let $V_n$ be the set of vertices reachable from $v_0$ by a walk of length $n$, defined as
\begin{align}
    V_n := \{ v \in V \mid (v,n) \in P \}.
\end{align}
Since $G$ is a connected non-bipartite graph, there exists some $N_0 \in \N$ such that $V_m = V$ for $m \geq N_0$. 
Assume that a compatible LCIF monoid structure $*$ exists on $(P,\leq)$. Fix $a = (u,n) \in P$, and consider the left multiplication map
\begin{align}
    L_a \co P \to U(a); \quad x \mapsto a * x.
\end{align}
Since $L_a$ is an order isomorphism from $P$ onto $U(a)$ and sends the minimum
element to $a$, it sends elements of rank $m$ to elements of rank $m+n$. 
Thus, for each $m \in \N$, there exists a map $\phi_{a,m} \co V_m \to V$ such that it can be written as
\begin{align}
    L_a(v,m) = (\phi_{a,m}(v), n + m).
\end{align}
Let us observe that $\phi_{a,m}$ is independent of $m$.

\begin{lem}
    There exists an automorphism $\phi_a \in \Aut(G)$ such that $\phi_{a,m} = \phi_a|_{V_m}$ for all $m \in \N$.
\end{lem}

\begin{proof}
    Since $L_a$ is bijective, for $a \in P, m \in \N$,
    \begin{align}
        \phi_{a,m}(v) = \phi_{a,m}(w) \iff L_a(v,m) = L_a(w,m) \iff (v,m) = (w,m),
    \end{align}
    which implies that $\phi_{a,m}$ is injective. In particular, when $m \geq N_0$, $\phi_{a,m}$ is bijective. 
    Moreover, for any $v\in V_m$ and $w\in V_{m+1}$, we have
    \begin{align}
        v \sim w &\iff (v,m) \lessdot (w,m+1) \\
        &\iff L_a(v,m) \lessdot L_a(w,m+1) \\
        &\iff (\phi_{a,m}(v), n + m) \lessdot (\phi_{a,m+1}(w), n + m + 1) \\
        &\iff \phi_{a,m}(v) \sim \phi_{a,m+1}(w).
    \end{align}
    In particular, when $m \geq N_0$,
    \begin{align}
        \forall v,w \in V, \quad v \sim w \iff \phi_{a,m}(v) \sim \phi_{a,m+1}(w),
    \end{align}
    so by \cref{lem:petersen-rigidity}, we have
    \begin{align}
        \phi_{a,N_0} = \phi_{a,N_0+1} = \cdots
    \end{align}
    Let this be denoted as $\phi_a$. Then, for any $v,w \in V$,
    \begin{align}
        v \sim w \iff \phi_a(v) \sim \phi_a(w),
    \end{align}
    implying $\phi_a \in \Aut(G)$.
    
    Next, we show by induction on $m$ that for each $0 \leq m < N_0$, $\phi_{a,m}$ is the restriction of $\phi_a$. 
    Assume $\phi_{a,m+1} = \phi_a|_{V_{m+1}}$. For any $v \in V_m$ and $w \in N_G(v) \subset V_{m+1}$, since
    \begin{align}
        \phi_{a,m}(v) \sim \phi_{a,m+1}(w) = \phi_a(w),
    \end{align}
    we have 
    \begin{align}
        \phi_a(N_G(v)) \subset N_G(\phi_{a,m}(v)).
    \end{align}
    On the other hand, since $\phi_a$ is a graph automorphism, 
    \begin{align}
        \phi_a(N_G(v)) = N_G(\phi_a(v)).
    \end{align}
    Therefore,
    \begin{align}
        N_G(\phi_a(v)) \subset N_G(\phi_{a,m}(v)),
    \end{align}
    and since both sides have the same number of elements (namely 3), the equality $N_G(\phi_a(v)) = N_G(\phi_{a,m}(v))$ holds. Since $G$ is twin-free, we obtain $\phi_{a,m}(v) = \phi_a(v)$. 
    Thus, by backward induction, we obtain $\phi_{a,m} = \phi_a|_{V_m}$ for $m = N_0-1,N_0-2,\dots,0$ as well.
\end{proof}

From the above, we see that the operation $*$ can be written using $\{\phi_a\}_{a\in P} \subset \Aut(G)$ as
\begin{align}
    a * (v,m) = (\phi_a(v), n + m).
\end{align}
By the associativity of $*$,
\begin{align}
    L_{a * b} = L_a \circ L_b \quad \therefore \phi_{a * b} = \phi_a \circ \phi_b.
\end{align}
The minimum element $(v_0,0)$ of $P$ is the identity element, so $\phi_{(v_0,0)} = \id$. Thus,
\begin{align}
    (P, *) \to (\Aut(G), \circ), \quad a \mapsto \phi_a
\end{align}
is a monoid homomorphism. Let $P_n = V_n \times \{n\}$ and let $m,n \geq N_0$. By the bijectivity of $\phi_a$, we have $a * P_m = P_{n+m}$ for any $a \in P_n$, which leads to
\begin{align}
    P_n * P_m = P_{n+m}.
\end{align}
Therefore, if we define $H_n = \{\phi_a \in \Aut(G) \mid a \in P_n\}$, then
\begin{align}
    H_n H_m = H_{n+m}.
\end{align}
Thus, $\{ H_n \mid n \geq N_0 \}$ forms a semigroup. Since a non-empty finite semigroup has an idempotent element, there exists some $N \geq N_0$ such that
\begin{align}
    H_N H_N = H_N.
\end{align}
Since a non-empty subsemigroup of a finite group is a subgroup, $H_N$ is a subgroup of $\Aut(G)$. Furthermore, for each $a = (v,N) \in P_N$, since
\begin{align}
    a * (v_0,0) = (v,N),
\end{align}
we get
\begin{align}
    \phi_a(v_0) = v.
\end{align}
Therefore, the action of $H_N$ on $V$ is transitive. Moreover, since $|H_N| \leq |V_N| = |V|$ by definition, the action is regular. 
Then \cref{prop:sabidussi} implies that $G$ is a Cayley graph, which contradicts \cref{prop:petersen_is_not_cayley}. 
From the above, we conclude that $P$ is not multiplicable. 
\end{proof}

\begin{rem}\label{rem:generalization}
    Similarly, if a graph $G$ satisfies the following conditions, then $P(G,v_0)$ is not multiplicable.
    \begin{enumerate}
        \item $G$ is a connected vertex-transitive graph.
        \item For any $\lambda \in \R$, $\lambda$ and $-\lambda$ are never both eigenvalues of the adjacency matrix of $G$, in particular, $0$ is not an eigenvalue.
        \item $G$ is not a Cayley graph.
    \end{enumerate}
    In fact, the properties of being twin-free and non-bipartite can be deduced from (2).
\end{rem}

\begin{ex}
    Let $k \ge 2$ be an integer. The Kneser graph $KG(2k+1, k)$ (also known as the Odd graph $O_{k+1}$) is defined as
    \begin{align}
        V = \binom{[2k+1]}{k}, \quad u \sim v \iff u \cap v = \emptyset.
    \end{align}
    This graph satisfies all the conditions required to form a non-multiplicable upho poset $P(KG(2k+1, k), v_0)$:
    \begin{enumerate}
        \item $KG(2k+1, k)$ is a connected vertex-transitive graph.
        \item The eigenvalues of the adjacency matrix of $KG(2k+1, k)$ are explicitly given by $\lambda_j = (-1)^j \binom{(2k+1)-k-j}{k-j} = (-1)^j (k - j + 1)$ for $j = 0, 1, \dots, k$. Since their absolute values $\{1, 2, \dots, k+1\}$ are all distinct, $\lambda$ and $-\lambda$ are never both eigenvalues.
        \item It is a well-known result by Godsil \cite{God80} that $KG(2k+1, k)$ is not a Cayley graph for $k \ge 2$.
    \end{enumerate}
    Note that when $k=2$, $KG(5,2)$ is exactly the Petersen graph.
\end{ex}

\section{Further Discussions}

\subsection{Multiplicability of the Line Graph of the Petersen Graph} \label{sec:line-graph}

Let $X := L(G)$ be the line graph of the Petersen graph $G$.
It is a non-Cayley vertex-transitive graph; see \cite[Theorem~3.6]{LS16}.
\begin{align}
    V(X) = E(G), \quad E(X) = \{\{e,f\} \subset E(G) \mid e \cap f \neq \emptyset\}.
\end{align}
Surprisingly, the upho poset $P(X,v_0)$ is multiplicable.

\begin{thm}\label{thm:line-graph-multiplicable}
    Let $X := L(G)$ be the line graph of the Petersen graph, and let $P := P(X, v_0)$. Then, $P$ is multiplicable.
    One such monoid structure can be described explicitly. First, we label the vertices of $X$ as follows. Here, $ij|kl$ denotes the vertex $\{\{i,j\}, \{k,l\}\}$ of $X$.
    \begin{align}
        v_{0} &= 15|23,\quad 
        v_{1} = 15|24,\quad 
        v_{2} = 12|35,\quad 
        v_{3} = 14|35,\quad 
        v_{4} = 12|45,\quad \\
        v_{5} &= 13|24,\quad 
        v_{6} = 25|13,\quad 
        v_{7} = 14|23,\quad 
        v_{8} = 34|15,\quad 
        v_{9} = 45|23,\quad \\
        v_{10} &= 34|25,\quad 
        v_{11} = 45|13,\quad 
        v_{12} = 35|24,\quad 
        v_{13} = 25|14,\quad 
        v_{14} = 12|34
    \end{align}
    Let $e_a, e_b, e_c, e_d$ be the atoms of $P$:
    \begin{align}
        e_a = (v_1,1),\ e_b = (v_9,1),\ e_c = (v_7,1),\ e_d=(v_8,1) \in P.
    \end{align}
    Then, the left multiplication by $e_s$ for $s \in \{a,b,c,d\}$ is given by
    \begin{align}
        e_s * (v_i,m) = \begin{cases}
            (v_{\sigma_s(i)}, m+1), & m \text{ is even} \\
            (v_{\tau_s(i)}, m+1), & m \text{ is odd}
        \end{cases}
    \end{align}
    Here, the permutations $\sigma_s,\tau_s$ are defined by the following table:
    \begin{table}[H]\centering
        \resizebox{\textwidth}{!}{
        \begin{tabular}{c|ccccccccccccccc}
            & 0 & 1 & 2 & 3 & 4 & 5 & 6 & 7 & 8 & 9 & 10 & 11 & 12 & 13 & 14 \\
            \hline
            $\sigma_a$ & 1 & 9 & 4 & 6 & 10 & 5 & 13 & 14 & 2 & 11 & 8 & 0 & 3 & 12 & 7 \\
            $\tau_a$ & 4 & 12 & 13 & 2 & 9 & 7 & 1 & 5 & 0 & 8 & 3 & 6 & 11 & 10 & 14 \\
            $\sigma_b$ & 9 & 6 & 12 & 11 & 0 & 1 & 10 & 13 & 8 & 14 & 5 & 3 & 7 & 2 & 4 \\
            $\tau_b$ & 10 & 0 & 9 & 3 & 2 & 13 & 12 & 4 & 11 & 7 & 14 & 8 & 5 & 6 & 1 \\
            $\sigma_c$ & 7 & 4 & 14 & 12 & 1 & 3 & 5 & 8 & 10 & 2 & 0 & 9 & 6 & 13 & 11 \\
            $\tau_c$ & 14 & 13 & 5 & 10 & 4 & 11 & 7 & 3 & 9 & 0 & 6 & 12 & 2 & 1 & 8 \\
            $\sigma_d$ & 8 & 5 & 3 & 1 & 14 & 2 & 11 & 6 & 4 & 9 & 12 & 13 & 10 & 7 & 0 \\
            $\tau_d$ & 11 & 14 & 8 & 13 & 7 & 6 & 3 & 0 & 1 & 10 & 9 & 4 & 12 & 5 & 2 \\
        \end{tabular}
        }
        \caption{Permutations defining the monoid structure on $P(L(G),v_0)$.} \label{tab:line-graph-permutations}
    \end{table}
\end{thm}

\begin{proof}
    We construct the multiplication using the canonical bipartite double cover of $X$. Let $\widehat{X} = X \times \Z/2$ be the bipartite graph defined as follows:
    \begin{align}
        V(\widehat{X}) &= V(X) \times \Z/2, \\
        E(\widehat{X}) &= \{{(v,\epsilon), (w,1-\epsilon)} \mid v \sim w \text{ in } X, \epsilon \in \Z/2\}.
    \end{align}
    For each $s\in\{a,b,c,d\}$, let $\sigma_s,\tau_s \in \mathfrak{S}_{15}$ be the permutations displayed in the table above. First, observe that
    \begin{align}
        v_i \sim v_j \iff v_{\sigma_s(i)}\sim v_{\tau_s(j)} \iff v_{\tau_s(i)}\sim v_{\sigma_s(j)} \quad \text{for } i,j \in \{0,1,\dots,14\}.
    \end{align}
    Since it preserves the adjacency relation, we can define an automorphism $\gamma_s \in \Aut(\widehat{X})$ by
    \begin{align}
        \gamma_s(v_i, 0) &= (v_{\sigma_s(i)}, 1), \quad \gamma_s(v_i, 1) = (v_{\tau_s(i)}, 0).
    \end{align}
    Moreover, by the choice of the four elements covering $(v_0,0)$, we have
    \begin{align}
        \gamma_a(v_0, 0)&=(v_1,1) = e_a,\quad
        \gamma_b(v_0, 0)=(v_9,1) = e_b, \\
        \gamma_c(v_0, 0)&=(v_7,1) = e_c,\quad
        \gamma_d(v_0, 0)=(v_8,1) = e_d.
    \end{align}
    Let $\Gamma$ be the group generated by $\gamma_a,\gamma_b,\gamma_c,\gamma_d$:
    \begin{align}
        \Gamma := \langle \gamma_a,\gamma_b,\gamma_c,\gamma_d\rangle \leq \Aut(\widehat X).
    \end{align}
    We define a monoid
    \begin{align}
        M := \left\langle (\gamma_a,1),(\gamma_b,1),(\gamma_c,1),(\gamma_d,1)\right\rangle \subset \Gamma \times \N
    \end{align}
    with the multiplication given by
    \begin{align}
        (\gamma, m) \cdot (\gamma', m') := (\gamma \circ \gamma', m + m').
    \end{align}
    The identity element is $(\id_{\widehat{X}},0)$.

    Since $\Gamma$ is a group, the monoid $\Gamma \times \N$ is cancellative on both sides. Therefore, its submonoid $M$ is also cancellative on both sides. Moreover, if $(\gamma, m) \cdot (\gamma', m') = (\id_{\widehat{X}}, 0)$, then $m + m' = 0$, which yields $m = m' = 0$. Since $(\gamma, 0), (\gamma', 0) \in M$, we must have $\gamma = \gamma' = \id_{\widehat{X}}$. Thus, $M$ is invertible-free.

    For $n \in \N$, let $M_n := M \cap (\Gamma \times \{n\})$ and $P_n := P \cap (V(X) \times \{n\})$. For a word $w = s_1 s_2 \cdots s_n$ with $s_k \in \{a,b,c,d\}$, we denote 
    \begin{align}
        \gamma_w := \gamma_{s_1} \circ \gamma_{s_2} \circ \cdots \circ \gamma_{s_n} \in \Gamma.
    \end{align}
    Since each $\gamma_s$ toggles the second coordinate of $V(\widehat{X})$, we can write
    \begin{align}
        \gamma_w(v_0, 0) = (v_i, n \bmod 2).
    \end{align}
    Let $\Phi_n \co M_n \to P_n$ be the map defined by 
    \begin{align}
        \Phi_n(\gamma_w, n) := (v_i,n).
    \end{align}
    Define $\Phi := \bigsqcup_{n \geq 0} \Phi_n \co M \to P$. We claim that each $\Phi_n$ is bijective. 

    \begin{lem}
        For each $n \in \N$, the map $\Phi_n \co M_n \to P_n$ is a bijection.
    \end{lem}

    \begin{proof}
        Let $A := \{\gamma_a, \gamma_b, \gamma_c, \gamma_d\} \subset \Gamma$. Then, $M_n = A^n \times \{n\}$. By explicit computation for $n = 0,1,\dots,6$, the cardinality of $A^n$ is:
        \begin{align}
            |A^0| &= 1,\ |A^1| = 4,\ |A^2| = 13,\ |A^3| = |A^4| = |A^5| = |A^6| = 15
        \end{align}
        Furthermore, $A^7 = A^3$ holds. Thus, $A^n = A^{n+4}$ for $n \geq 3$, which implies:
        \begin{align}
            |M_0| &= 1, \quad |M_1| = 4, \quad |M_2| = 13, \quad |M_n| = 15 \text{ for } n \geq 3.
        \end{align}
        On the other hand, we can also compute $|P_n|$ for $n = 0,1,2,3$ as follows:
        \begin{align}
            |P_0| &= 1, \quad |P_1| = 4, \quad |P_2| = 13, \quad |P_n| = |V(X)| = 15 \text{ for } n \geq 3.
        \end{align}
        This shows that $|M_n| = |P_n|$ for all $n \in \N$. Therefore, it suffices to prove that $\Phi_n$ is injective for each $n$.

        Let $K := \operatorname{Stab}_\Gamma(v_0, 0)$. Suppose $(\gamma,n), (\delta,n) \in M_n$ satisfy $\Phi_n(\gamma,n) = \Phi_n(\delta,n)$. Then $\gamma(v_0, 0) = \delta(v_0, 0)$, which means $\delta^{-1}\gamma \in K$. 
        
        Since $A^7=A^3$, we have $A^{n+4}=A^n$ for all $n\geq 3$.
        Hence it is enough to verify the following condition for $n=0,1,\dots,6$:
        \begin{align}
            (A^n)^{-1}A^n\cap K=\{\id_{\widehat X}\}.
        \end{align}
        This finite verification proves the injectivity of $\Phi_n$ for all $n$.
    \end{proof}

    Let $x=(\gamma,n)\in M_n$ and suppose $\Phi(x)=(v_i,n)$. For any $s\in\{a,b,c,d\}$, we have
    \begin{align}
        x \cdot (\gamma_s,1)=(\gamma\gamma_s,n+1).
    \end{align}
    Since $\gamma_s(v_0, 0)\in N_{\widehat{X}}(v_0, 0)$ and $\gamma$ is an automorphism of $\widehat{X}$, it follows that
    \begin{align}
        \gamma \gamma_s (v_0, 0) \in N_{\widehat{X}}(\gamma (v_0, 0)).
    \end{align}
    Therefore, $\Phi(x)\lessdot \Phi(x \cdot (\gamma_s,1))$. Thus, $\Phi$ maps the covering relations of $M$ to those of $P$.

    Conversely, since $X=L(G)$ is a 4-regular graph, each element in $P$ covers exactly 4 elements. Meanwhile, the covers of $x$ in $M$ are exactly the 4 elements $x \cdot (\gamma_s,1)$ for $s \in \{a,b,c,d\}$, which are distinct due to left-cancellativity. 
    Since both orders are the transitive closures of their covering relations and
    $\Phi$ is rank-preserving and bijective, this implies that $\Phi$ is an order isomorphism.

    Finally, if we define the product $*$ on $P$ by
    \begin{align}
        p * q := \Phi(\Phi^{-1}(p) \cdot \Phi^{-1}(q)),
    \end{align}
    then $(P,*)$ becomes an LCIF monoid isomorphic to $M$, establishing a multiplicability on $P$. Furthermore, for $(v_i,m) \in P$, writing $\Phi^{-1}((v_i,m)) = (\gamma, m)$, we obtain:
    \begin{align}
        e_s * (v_i,m) &= \Phi((\gamma_s,1) \cdot (\gamma, m)) \\
        &= \Phi(\gamma_s \gamma, m+1).
    \end{align}
    Since
    \begin{align}
        \gamma_s \gamma(v_0, 0) = \gamma_s(v_i, m \bmod 2) = \begin{cases}
            (v_{\sigma_s(i)}, 1), & m \text{ is even} \\
            (v_{\tau_s(i)}, 0), & m \text{ is odd}
        \end{cases}
    \end{align}
    it precisely follows that
    \begin{align}
        e_s * (v_i,m) = \begin{cases}
            (v_{\sigma_s(i)}, m+1), & m \text{ is even} \\
            (v_{\tau_s(i)}, m+1), & m \text{ is odd}
        \end{cases}
    \end{align}
    completing the proof.
\end{proof}

\begin{rem}
    The eigenvalues of the adjacency matrix of $L(G)$ are 
    \begin{align}
        4, 2^{(5)}, (-1)^{(4)}, (-2)^{(5)}
    \end{align}
    so $L(G)$ does not satisfy the condition (2) mentioned in \cref{rem:generalization}. 
    Indeed, the permutations $\sigma_s$ and $\tau_s$ defined above show that \cref{lem:petersen-rigidity} does not hold for $L(G)$.
\end{rem}

\begin{rem}
    As an abstract monoid, $M$ is isomorphic to the monoid generated by $a,b,c,d$ subject to the following 31 relations:
    \begin{align}
        ba &= ad,\ cb = ad,\ dc = ad,\ bbc = abd,\ bbd = aaa,\ bca = aab,\\
        bcc &= acd,\ bcd = abb,\ bdb = aac,\ bdd = acc,\ caa = acc,\ cab = abc,\\
        cac &= aaa,\ cad = ada,\ cca = bbb,\ ccc = aca,\ ccd = aab,\ cda = acd,\\
        cdb &= abd,\ cdd = bda,\ daa = abc,\ dab = bda,\ dac = aab,\ dad = adb,\\
        dbb &= acc,\ dbc = abb,\ dbd = bbb,\ dda = abd,\ ddb = aca,\ ddd = aac,\\
        bbbb &= aacd
    \end{align}
\end{rem}

\subsection{A monoid with the same rank-generating function}

The Petersen example shows that the walk poset $P(G,v_0)$ itself is not
multiplicable. However, this does not rule out the possibility that its
rank-generating function is realized by a multiplicable upho poset, or
equivalently by a finitely generated LCH monoid.

Indeed, the rank-generating function of $P(G,v_0)$ is
\begin{align}
    F_P(t)=1+3t+7t^2+9t^3+10t^4+10t^5+\cdots.
\end{align}
The following example shows that this series is realized by a finitely
generated LCH monoid.

\begin{ex}\label{ex:same-generating-function}
    Let $M$ be the monoid generated by $a,b,c$ subject to the following relations:
    \begin{align}
        ba &= ab, &
        bb &= aa, &
        bca &= acb, &
        bcb &= aca, \\
        caa &= aab, &
        cab &= aaa, &
        cac &= abc, &
        cbc &= aac, \\
        cca &= acb, &
        ccb &= aca, &
        cccc &= bccc.
    \end{align}
    A finite rewriting verification shows that $M$ is left-cancellative and homogeneous, and that its elements have normal forms
    \begin{align}
        a^k v \qquad (k\in\N,\ v\in V),
    \end{align}
    where
    \begin{align}
        V=\{e,b,c,bc,ca,cb,cc,bcc,ccc,bccc\}.
    \end{align}
    Consequently,
    \begin{align}
        \sum_{n\geq 0}|M_n|t^n = 1+3t+7t^2+9t^3+10t^4+10t^5+\cdots = F_P(t).
    \end{align}
\end{ex}

\subsection{Open problems}

The Petersen example shows that not every finite-type $\N$-graded upho poset
is multiplicable. Hence the regularity conjecture of Fu--Peng--Zhang fails
at the level of posets. Nevertheless, two weaker or restricted forms of
multiplicability remain plausible.

First, our counterexample is not a lattice. It is therefore natural to ask
whether the regularity phenomenon survives in the lattice setting.

\begin{conj}\label{conj:upho-lattice-multiplicable}
    Every finitary upho lattice, and in particular every finite-type $\N$-graded upho lattice, is multiplicable.
\end{conj}

This conjecture is deeply connected to the classification theory of upho lattices via the concept of ``\textbf{cores}'' introduced by Hopkins \cite{Hop25} and Hopkins and Lewis \cite{HL26}.
In their theory, the core of an upho lattice is a finite graded lattice that completely determines its rank-generating function.
A particularly striking consequence of their framework is that, for any given core $L$, there can exist only a finite number of multiplicable upho lattices. 
While it currently remains an open question whether the total number of general upho lattices associated with a given core is finite, \cref{conj:upho-lattice-multiplicable}, if true, would elegantly resolve this problem. 
Specifically, it would immediately imply that the structural classification of upho lattices is entirely governed by the finite behavior of their cores.

Second, even if a finite-type $\N$-graded upho poset is not multiplicable
itself, its rank-generating function may still be realized by a multiplicable
upho poset, or equivalently by a finitely generated LCH monoid, as in \cref{ex:same-generating-function}. 

\begin{conj}\label{conj:upho-rank-series-monoid}
    Let $P$ be a finite-type $\N$-graded upho poset. Then there exists
    a finitely generated LCH monoid $M$ such that
    \begin{align}
        |P_n| = |M_n| \quad \text{for all } n \in \N.
    \end{align}
    Here $P_n$ denotes the set of elements of rank $n$ in $P$, and $M_n$ denotes the set of elements of length $n$ in $M$.
\end{conj}

This suggests the following broader enumerative problem.

\begin{ques}\label{ques:rank-sequence-constraints}
    Which sequences
    \begin{align}
        1 = r_0, r_1, r_2, \ldots
    \end{align}
    occur as the rank sequence
    \begin{align}
        r_n = |P_n|
    \end{align}
    of a finite-type $\N$-graded upho poset, and which of these can be realized by a finitely generated LCH monoid?
\end{ques}

Finally, our construction raises the following graph-theoretic problem.

\begin{ques}\label{ques:graph-characterization}
    For which finite vertex-transitive graphs $G$ is the walk poset $P(G,v_0)$ multiplicable?
\end{ques}

The Petersen graph gives a negative example, whereas the line graph of the
Petersen graph gives a positive example. Thus non-Cayleyness of $G$ alone
does not determine the multiplicability of $P(G,v_0)$.

\appendix

\section{Z3 search and finite certificates}
\label{app:z3-search}

This appendix describes the finite Z3 searches used to find some of the explicit data appearing in the paper. 
The search procedures themselves are not part of the conceptual proofs.  Once the finite data are obtained, the
arguments in the main text use them as certificates: the required properties are reduced to explicit finite checks.

We used the Z3 SMT solver~\cite{Z3} in three places.  First, we searched for compatible partial LCIF multiplication tables on finite truncations of walk posets.
Second, from a satisfying multiplication table for the line graph of the
Petersen graph, we extracted the explicit data used in
\cref{thm:line-graph-multiplicable}.  Third, we used a similar search
to find the monoid in \cref{ex:same-generating-function}.

\subsection{Encoding finite truncations of walk posets}

We encode finite truncations of walk posets as follows.
Let $P_{\leq H}$ be the truncation of a walk poset up to rank $H$.
We introduce an auxiliary symbol $\bot$, which represents products leaving
the truncation.  The unknown multiplication is encoded by a function
\begin{align}
    M:(P_{\leq H}\cup\{\bot\})^2\to P_{\leq H}\cup\{\bot\}.
\end{align}
The constraints are:
\begin{enumerate}
    \item $M(\bot,x)=M(x,\bot)=\bot$ for all $x$;
    \item $M(e,x)=M(x,e)=x$ for all $x$;
    \item If $x$ has rank $<H$, then $\{M(x,a)\mid a\in A\}=\{y\in P_{\leq H}\mid x\lessdot y\}$, where $A$ is the set of atoms.  If $x$ has rank $H$, then $M(x,a)=\bot$ for all $a\in A$.
    \item associativity is imposed whenever the rightmost factor is an atom; $M(x,M(y,a))=M(M(x,y),a)$ for all $x,y$ and atoms $a$;
    \item $M(x,y_1)=M(x,y_2)\neq \bot$ implies $y_1=y_2$ for all $x,y_1,y_2$.
\end{enumerate}
We denote this finite SMT instance by $\Phi_H$.

\subsection{Search for the line graph multiplication}
\label{app:line-graph-search}

For the line graph $L(G)$ of the Petersen graph, the above SMT instance
with $H=4$ was satisfiable.  From a satisfying model, we extracted a partial
multiplication table.  We then converted this table into the explicit finite
data used in \cref{sec:line-graph}: the generators correspond to the
atoms, and the products by generators determine the required maps.

The explicit data displayed in \cref{tab:line-graph-permutations} should
therefore be regarded as a certificate found by Z3.  The proof of
\cref{thm:line-graph-multiplicable} does not depend on the search
procedure; it only uses the finite identities verified from this table.

\subsection{Search for a monoid with the same rank-generating function}
\label{app:same-gf-search}

We also used a finite-search strategy to look for homogeneous presentations
whose initial degree sequence begins with
\begin{align}
    1,3,7,9,10,10,\ldots.
\end{align}
A satisfying model was converted into a finite presentation by recording
relations whenever two words represented the same element in the model.
This procedure led to the presentation in \cref{ex:same-generating-function}.

\subsection{The Petersen obstruction}
\label{app:petersen-z3}

For the Petersen graph itself, the instance $\Phi_4$ was unsatisfiable.
Since any compatible LCIF monoid structure on the full walk poset would
restrict to a satisfying assignment of $\Phi_H$ for every $H$, this gives
a finite obstruction to multiplicability.  This computation agrees with the
theoretical proof of \cref{thm:petersen-nonmultiplicable}.

\section*{Acknowledgements}

The author is deeply indebted to Hidetaka Inagaki for his substantial contribution to this work. He not only inspired this project by independently arriving at the conjecture but also carefully read an earlier draft and pointed out a critical error in the proof, saving the author from a serious mistake.
The author also thanks Yui Ushiro and Manato Shimizu for collecting examples that provided crucial insights.  
The author is grateful to Masahiro Hachimori for his valuable comments on the draft, which significantly improved the consistency with the previous literatures.
Finally, the author used Gemini for preliminary brainstorming. All mathematical arguments, computations, and final writing were checked by the author.

\end{document}